\documentclass[11pt]{amsart}
\usepackage{amsthm, amsmath,amsfonts, amssymb, enumerate}
\usepackage{graphicx,tikz}

\newcommand{\N}{\mathbb{N}}

\newtheorem{thm}{Theorem}[section]
\newtheorem{lem}[thm]{Lemma}

\newtheorem{cor}[thm]{Corollary}
\newtheorem{que}[thm]{Question}

\DeclareMathOperator{\CAT}{CAT}
\DeclareMathOperator{\id}{id}
\DeclareMathOperator{\im}{im}
\newcommand{\Geo}[2]{[#2, #1)}
\DeclareMathOperator{\Pow}{Pow}

\theoremstyle{definition}
\newtheorem{defn}[thm]{Definition}

\begin{document}
\title[Boundary actions of hyperbolic groups]{Hyperfiniteness of boundary actions of cubulated
  hyperbolic groups}
\author{Jingyin Huang}
\address{Jingyin Huang, McGill University, Department of Mathematics and
  Statistics, 805 Sherbrooke Street W, Montreal, QC, H3A 0B9
Canada} 
\email{jingyin.huang@mail.mcgill.ca}
\author{Marcin Sabok}
\address{Marcin Sabok, McGill University, Department of Mathematics and
  Statistics, 805 Sherbrooke Street W, Montreal, QC, H3A 0B9
Canada and Instytut Matematyczny PAN, \'{S}niadeckich 8,
00-656 Warszawa, Poland}
\email{marcin.sabok@mcgill.ca}
\author{Forte Shinko}
\address{Forte Shinko, McGill University, Department of Mathematics and
  Statistics, 805 Sherbrooke Street W, Montreal, QC, H3A 0B9
Canada}
\email{forte.shinko@mail.mcgill.ca}
\thanks{The authors would like to acknowledge support from
  the NCN (Polish National Science Centre) through the grant
  \textit{Harmonia} no. 2015/18/M/ST1/00050. Marcin Sabok
  acknowledges also support from 
  NSERC through the \textit{Discovery Grant} RGPIN-2015-03738}

\maketitle
\begin{abstract}
  We show that if a hyperbolic group acts geometrically on a
  CAT(0) cube complex, then the induced boundary action is
  hyperfinite. This means that for a cubulated
  hyperbolic group the natural action on its Gromov boundary
  is hyperfinite, which generalizes an old result of
  Dougherty, Jackson and Kechris for the free group case.
\end{abstract}

\section{Introduction}

The complexity theory for countable Borel equivalence
relations has been an active topic of study over the last
few decades. By the classical result of Feldman and Moore \cite{feldman.moore},
countable Borel equivalence relations correspond to Borel
actions of countable groups and there has been a lot of
effort to understand how the structure of the actions of a group
depends on the group itself.

Recall that if $Z$ is a standard Borel space,
then a \textit{Borel equivalence relation} on $Z$
is an equivalence relation $E\subseteq Z^2$ which is Borel in $Z^2$.
If $E$ and $F$ are Borel equivalence relations on $Z$ and $Y$ respectively,
we say that $E$ is \textit{Borel-reducible} to $F$ (denoted $E\le_B F$)
if there is a Borel function $f:Z\to Y$ such that
$z_1\mathrel{E}z_2$ if and only if  $f(z_1)\mathrel{F}
f(z_2)$ for all $z_1,z_2\in Z$
($f$ is then called a \textit{reduction} from $E$ to $F$).
A \textit{smooth} equivalence relation
is a Borel equivalence relation which is reducible to $\id_{2^\N}$,
the equality relation on the Cantor set. The relation $E_0$
is defined on the Cantor set $2^\N$ as follows:
$x\mathrel{E_0}y$ if there exists $n$ such that $x(m)=y(m)$
for all $m>n$.
A \textit{finite} (resp. \textit{countable}) equivalence relation
is an equivalence relation whose classes are finite (resp. countable).
An equivalence relation $E$ on $X$ is \textit{hyperfinite}
(resp. \textit{hypersmooth})
if there is a sequence $F_n$ of finite (resp. smooth) equivalence relations on $X$
such that $F_n\subseteq F_{n+1}$ and $E = \bigcup_n F_n$.
Note that if $E\le_B F$ and $F$ is hypersmooth,
then $E$ is also hypersmooth. 

 Among countable
equivalence relations, hyperfinite equivalence relations are
exactly those which are Borel-reducible to $E_0$ \cite{djk}. The
classical dichotomy of Harrington, Kechris and Louveau \cite{hkl}
implies that if a countable Borel equivalence relation is
not smooth, then $E_0$ is Borel-reducible to it. Interestingly, a very recent result of Conley and
Miller \cite{conley.miller} implies that among countable Borel equivalence
relations which are not hyperfinite there is no countable basis
with respect to Borel-reducibility.

Hyperfinite equivalence relations have particular structure,
observed by Slaman and Steel and independently by Weiss
(see \cite[Theorem 7.2.4]{gao}). An equivalence relation $E$ on $Z$ is hyperfinite if and
only if there exists a Borel action of the group of integers
$\mathbb{Z}$ on $Z$ which induces $E$ as its orbit
equivalence relation. In recent years, there has been a lot
of effort to understand which groups induce hyperfinite
equivalence relations. For instance, Gao and Jackson \cite{gj}
showed that Borel actions of all Abelian groups induce
hyperfinite equivalence relations. It is still unknown if
all Borel actions of amenable groups induce hyperfinite
equivalence relations.

In this paper we are mainly interested in actions of
hyperbolic groups. Recall that a geodesic metric space $X$
is \textit{hyperbolic} if there exists $\delta>0$ such that
all geodesic triangles in $X$ are \textit{$\delta$-thin},
i.e. each of their sides is contained in the
$\delta$-neighborhood of the union of the other two
sides. In such case we also say that $X$ is
\textit{$\delta$-hyperbolic}.  A finitely generated group
$G$ is \textit{hyperbolic} if its Cayley graph (w.r.t. an
arbitrary finite generating set) is hyperbolic. An isometric
action of a group $G$ on a metric space $X$ is
\textit{proper} if for every compact subset $K\subseteq X$
the set $\{g\in G: gK\cap K\neq \emptyset \}$ is finite.  An
isometric action of $G$ on $X$ is \textit{cocompact} if
there exists a compact subset $A$ of $X$ such that $GA =
X$.
If $X$ is a combinatorial complex, then an isometric action
of a group on $X$ is proper if and only if the stabilizers
of all vertices are finite. Similarly, an action on a
combinatorial complex is cocompact if and only if there are
finitely many orbits of vertices. An action of a group is
called \textit{geometric} if it is both proper and
cocompact. If a group $G$ acts geometrically on a geodesic
metric space $X$ by isometries, then $G$ is hyperbolic if
and only if $X$ is hyperbolic, since hyperbolicity is
invariant under quasi-isometries.

Given a geodesic hyperbolic space $X$ we denote by $\partial X$ its
\textit{Gromov boundary} (for definition see e.g. \cite{kapovich}). Any geometric
action of a hyperbolic group $G$ on a hyperbolic space $X$ induces a
natural action of $G$ on $\partial X$ by homeomorphisms.
If $X$ is the Cayley graph of a hyperbolic group $G$, then
the Gromov boudary of $X$ is called the \textit{Gromov
  boundary of the group} $G$. 

Hyperbolic groups often admit geometric actions on CAT$(0)$
cube complexes. Recall that a \emph{cube complex} is obtained by taking a disjoint 
collection of unit cubes in Euclidean spaces of various 
dimensions, and gluing them isometrically along their 
faces. A geodesic metric space $X$ is
\textit{CAT$(0)$} if for every geodesic triangle $\Delta$ in $X$ and
a comparison triangle $\Delta'$ in the Euclidean plane, with
sides of the same length as the sides of $\Delta$, the
distances between points on $\Delta$ are less than or equal
to the distances between the corresponding points on
$\Delta'$. This is one way of saying that a metric space has
nonpositive curvature. For more details on CAT$(0)$ cube complexes see
Section \ref{sec:cube-compl}.

If a hyperbolic group admits a geometric action on a
CAT$(0)$ cube complex, then we say that it is
\textit{cubulated}. Examples of cubulated hyperbolic groups
include 
\begin{itemize}
\item fundamental groups of hyperbolic surfaces and
hyperbolic closed 3-manifolds (Kahn and Markovic
\cite{kahn.markovic} and Bergeron and Wise
\cite{bergeron.wise}),
\item uniform hyperbolic lattices of
"simple type" (Haglund and Wise \cite{haglund.wise}), 
\item hyperbolic Coxeter groups (Niblo and Reeves
\cite{niblo.reeves} and Caprace and M\"uhlherr
\cite{caprace.muhlherr}),
\item $C'(1/6)$ or $C'(1/4)$-$T(4)$
metric small cancellation groups (Wise \cite{wise.cub}),
\item certain cubical small cancellation groups (Wise
\cite{wise}),
\item Gromov's random groups with density $<1/6$
(Ollivier and Wise \cite{ollivier.wise}),
\item hyperbolic
free-by-cyclic groups (Hagen and Wise \cite{hagen.wise.irr}
in the irreducible case and \cite{hagen.wise.gen} in the
general case).
\end{itemize}
It is worth noting that cubulations of
hyperbolic groups played important role in recent
breakthroughs on the Virtual Haken Conjecture by Agol
\cite{agol} and Wise \cite{wise}. The main result of this paper is the
following.

\begin{thm}\label{main}
   If a hyperbolic group $G$ acts geometrically on a
   $\mathrm{CAT}(0)$ cube complex $X$, then the induced action on
   $\partial X$ is hyperfinite.
\end{thm}

Note that if $G$ acts geometrically on $X$ and $Y$, then there is a
$G$-equivariant homeomorphism of $\partial X$ and $\partial
Y$ \cite{gromov}. Hence, the above theorem implies the following.

\begin{cor}\label{corol}
  If $G$ is a hyperbolic cubulated group, then its natural
  boundary action on $\partial G$ is hyperfinite.
\end{cor}

The boundary actions of hyperbolic groups
have been studied from the perspective of their
complexity. Recall that if $\mu$ is a probability measure on
a standard Borel space $X$ and $E$ is a countable Borel
equivalence relation on $X$, then $E$ is called
\textit{$\mu$-hyperfinite} if there exists a $\mu$-conull
set $A\subseteq X$ such that $E\cap A^2$ is
hyperfinite. A Borel probability measure $\mu$ is $E$-\textit{quasi-invariant} if it is quasi-invariant with
respect to any group action inducing $E$. Kechris and Miller
\cite[Corollary 10.2]{kechris.miller} showed that if $E$ is
$\mu$-hyperfinite for all $E$-quasi-invariant Borel
measures, then $E$ is $\mu$-hyperfinite for all Borel
measures $\mu$. It is worth noting, however, that for
boundary actions of hyperbolic groups usually there is no
unique quasi-invariant measure on the boundary.

In the case of the free group, its boundary action
induces the equivalence relation which is Borel bi-reducible
with the so-called \textit{tail equivalence relation} on the
Cantor set: $x\mathrel{ E_t} y$ if
$\exists n\,\exists m\,\forall k\ x(n+k)=y(m+k)$. It follows
from the results of Connes Feldman and Weiss \cite[Corollary
13]{cfw} and Vershik \cite{vershik}
that if $G$ is the free group, then the action of $G$ on its
Gromov boundary (which is the Cantor set) is
$\mu$-hyperfinite for every Borel quasi-invariant
probability measure. Dougherty Jackson and Kechris
\cite[Corollary 8.2]{djk} showed
later that the tail equivalence relation is actually
hyperfinite. On the other hand, Adams \cite{adams} showed that for every hyperbolic group
$G$ the action on $\partial G$ is $\mu$-hyperfinite for all
Borel quasi-invariant probability measures. We do not know
how to generalize Corollary \ref{corol} to all hyperbolic groups.

Our proof uses an idea of Dougherty,  Jackson and Kechris \cite{djk}
and the main ingredient of the proof is a result that seems
to be interesting on its own right. Given an element a
hyperbolic group $G$ acting geometrically on a cube complex $C$, 
$\gamma\in\partial C$ and an element $x\in C$, define the \textit{interval}
$\Geo{\gamma}{x}$ to be the set of all vertices of the complex
which lie on a geodesic ray in the 1-skeleton of
$C$ from
$x$ to $\gamma$. We would like to emphasize here that we
consider here only the 1-skeleton of $C$ and all geodesics
we consider are the \textit{combinatorial geodesics}, i.e. those
taken in the 1-skeleton.

\begin{lem}\label{modulo.finite}
  If a hyperbolic group $G$ acts geometrically on a
  $\mathrm{CAT}(0)$ cube complex $C$ and $\gamma\in \partial C$, then
  for every $x,y\in C$ the sets $\Geo{\gamma}{x}$ and
  $\Geo{\gamma}{y}$ differ by a finite set. 
\end{lem}

Theorem \ref{main} is obtained using Lemma
\ref{modulo.finite} and the following result:

\begin{thm}\label{abstractthm}
  Suppose a hyperbolic group $G$ acts freely and cocompactly
  on a locally finite graph $V$ such that for every 
  $\gamma\in \partial V$ and
  for every $x,y\in V$ the sets $\Geo{\gamma}{x}$ and
  $\Geo{\gamma}{y}$ differ by a finite set.  Then the
  action of $G$ on $\partial V$ induces a hyperfinite
  equivalence relation.
\end{thm}

The assumption that $G$ acts freely on $V$ means
that the action on the set of vertices of $V$ is free.

Given a fixed finite set of generators for a hyperbolic
group $G$, the group acts on its Cayley graph. For $g\in G$
and $\gamma\in\partial G$, the set
$\Geo{\gamma}{g}$ is defined as above.
The following question seems natural. \footnote{it has been
  answered recently in the negative by N. Touikan}

\begin{que}
  Suppose $G$ is a hyperbolic group with a fixed finite
  generating set and $\gamma\in\partial
  G$. Is it true that for any two group elements $g,h\in G$
  the sets $\Geo{\gamma}{g}$ and $\Geo{\gamma}{h}$ differ by a
  finite set?
\end{que}

Of course, it may turn out that the answer to the above
question depends on the choice of the generating set. Or,
more generally, one can ask the following question.

\begin{que}
 Is it true that for every hyperbolic group $G$ there exists a
 locally finite graph $V$ such that $G$ acts geometrically (or even freely
 and cocompactly) on
 $V$, and $V$ has the property that
 $\Geo{\gamma}{x}$ and $\Geo{\gamma}{y}$ have finite symmetric
 difference for every $\gamma\in\partial V$ and $x,y\in V$?
\end{que}

We should add here that the class of groups for which we
 can prove the positive answer to the above question is
 limited to groups with the Haagerup property. We do not know
 any examples of groups which have the property stated
 in the above question and do not have the Haagerup property.

\medskip

\textbf{Acknowledgement}.
We would like to thank Piotr Przytycki for inspiration and
many helpful discussions.

\section{$\mathrm{CAT}(0)$ cube complexes}\label{sec:cube-compl}

Here we give a summary of several basic properties of CAT$(0)$ cube 
complexes without proof. For more details we refer the 
reader to \cite[Chapter II.5]{bridson2011metric} and 
\cite{sageev2012cat}. 

Recall that a \emph{cube complex} is obtained by taking a
disjoint collection of unit cubes in Euclidean spaces of
various dimensions, and gluing them isometrically along
their faces. In particular, every cube complex has a
piecewise Euclidean metric.

A cube complex $X$ is \emph{uniformly locally finite} if 
there exists $D>0$ such that each vertex is contained in at 
most $D$ edges. Note that if $X$ admits a cocompact group 
action, then it is automatically uniformly locally finite. 

Now, for each vertex $v$ in a cube complex $X$, draw an 
$\varepsilon$-sphere $S_{v}$ around $v$. Note that the cubes of 
$X$ divide $S_v$ into simplices (a priori, these simplices may 
not be embedded in $S_v$, since a cube may not be embedded 
in $X$). Thus $S_v$ has the structure of a combinatorial 
cell complex which is made of various simplices glued along 
their faces. This complex is called the \emph{link} of the
vertex $v$. 

Recall that a simplicial complex $K$ is \emph{flag} if every 
complete subgraph of the 1-skeleton of $K$ is actually the 
1-skeleton of a simplex in $K$. 

\begin{defn}
	\label{defn:ccc}
        A \emph{$\CAT(0)$ cube complex} is a cube complex 
        which is simply connected and such that the link of 
        each its vertex is a flag simplicial complex. 
\end{defn}

The above is a combinatorial equivalent definition of
CAT$(0)$ property for cube complexes (for more details see
\cite[Definition II.1.2]{bridson2011metric}).

Let $X$ be a CAT$(0)$ cube complex with its piecewise 
Euclidean metric. A subset of $C\subseteq X$ is 
\textit{convex} if for any two points $x,y\in C$, any 
geodesic segment connecting $x$ and $y$ is contained in 
$C$. A \emph{convex subcomplex} of $X$ is a subcomplex which 
is also convex. 

Recall that a \textit{mid-cube} of $C=[0,1]^{n}$ is a subset
of form $f_{i}^{-1}(\{\frac{1}{2}\})$, where $f_{i}$ is one of the
coordinate functions.

\begin{defn}
  A \textit{hyperplane} $h$ in $X$ is a subset such that 
  \begin{enumerate}
  \item $h$ is connected. 
  \item For each cube $C\subseteq X$, $h\cap C$ is either 
    empty or a mid-cube of $C$. 
  \end{enumerate}
 \end{defn}

 It was proved by Sageev \cite{sageev1995ends} that for each 
 edge $e\in X$, there exists a unique hyperplane which 
 intersects $e$ in one point. This is called the hyperplane 
 \textit{dual} to the edge $e$. Actually, given an edge $e$,
 we can always build locally a piece of hyperplane that cuts 
 through $e$. In order to extend this piece to a hyperplane,
 one needs to make sure that the this piece does not run into 
 itself when one extends it. It is shown in 
 \cite{sageev1995ends} that this can never happen in a 
 CAT$(0)$ cube complex and thus such extensions exist.

 Let $X$ be a CAT$(0)$ cube complex, and let $e\subseteq X$ be
 an edge. Denote the hyperplane dual to $e$ by $h_{e}$. The
 following facts about hyperplanes are well-known
 \cite{sageev1995ends,sageev2012cat}.
\begin{enumerate}
\item The hyperplane $h_{e}$ is a convex subset of $X$ and
  $h_{e}$ with the induced cell structure from $X$ is also a
  CAT$(0)$ cube complex.
\item $X\setminus h_{e}$ has exactly two connected
  components, which are called \textit{halfspaces}.
\end{enumerate}

Two points in $X$ are \emph{separated} by a hyperplane $h$
if they are different connected components of
$X\setminus h$.


We use the following metric on the 0-skeleton $X^{(0)}$ of
$X$. Given two vertices in $X^{(0)}$, the
$\ell^{1}$-distance between them is defined to be the length
of the shortest path joining them in the 1-skeleton
$X^{(1)}$.  By \cite[Lemma 13.1]{MR2377497}, the
$\ell^{1}$-distance between any two vertices is equal to
number of hyperplanes separating them.

Given two vertices $u,v\in X$, a \textit{combinatorial
  geodesic} between them is an edge path in $X^{(1)}$
joining $u$ and $v$ which realizes the $\ell^{1}$-distance
between $u$ and $v$. Note that there may be several different
combinatorial geodesics joining $u$ and $v$.  By \cite[Lemma
13.1]{MR2377497}, an edge path $\omega\subseteq X^{(1)}$ is a
combinatorial geodesic if and only if for each pair of
different edges $e_1,e_2\subseteq\omega$, the hyperplane dual
to $e_1$ and the hyperplane dual to $e_2$ are different.

An edge path $\omega$ \textit{crosses} a hyperplane
$h\subseteq X$ if there exists an edge $e\subseteq\omega$ such
that $h$ is the dual to $e$. So, in other words, $\omega$ is
a combinatorial geodesic if and only if there does not exist
a hyperplane $h\subseteq X$ such that $\omega$ crosses $h$
more than once.

Let $Y\subseteq X$ be a convex subcomplex (with respect to the
piecewise Euclidean metric). Then, by \cite[Proposition
13.7]{MR2377497}, $Y$ is also convex with respect to the
$\ell^{1}$-metric in the following sense: for any vertices
$u,v\in Y^{(0)}$, every combinatorial geodesic joining $u$
and $v$ is contained in $Y$.

In the rest of this paper, we will always use the
$\ell^1$-metric on $X^{(0)}$ and use $d$ to denote this
metric.

Let $Y\subseteq X$ be a convex subcomplex. By \cite[Lemma
13.8]{MR2377497}, for any vertex $v\in X$, there exists a
unique vertex $u\in Y$ such that $d(u,v)=d(v,Y^{(0)})$. Thus,
we have a nearest point projection map
$\pi_Y:X^{(0)}\to Y^{(0)}$.

\begin{lem}
	\label{l:45}
        Let $Y\subseteq X$ be a convex subcomplex and $v\in
        X$. Let $\omega$ be a combinatorial geodesic from $v$ to
        a vertex in $Y$ which realizes the $\ell^1$ distance
        between $v$ and $Y^{(0)}$. Then each hyperplane dual
        to an edge in $\omega$ separates $v$ from
        $Y$. Conversely, each hyperplane which separates $v$ from
        $Y$ is dual to an edge in $\omega$.
\end{lem}
\begin{proof}
  This is a special case of \cite[Proposition 13.10]{MR2377497}.
\end{proof}

The following is a consequence of Lemma
\ref{l:45} and the fact that  the $\ell^{1}$-distance
between any two vertices is equal to number of hyperplanes
separating them \cite[Lemma 13.1]{MR2377497}.

\begin{cor}
\label{c:2}
Let $Y\subseteq X$ be a convex subcomplex. For every $v\in
X$ the distance $d(v,Y^{(0)})$ is the number of hyperplanes that separate $v$ from $Y$.
\end{cor}

\begin{lem}
\label{l:5}
Let $Y\subseteq X$ be a convex subcomplex and
$\pi_Y:X^{(0)}\to Y^{(0)}$ be the nearest point
projection. Given two adjacent vertices $u,v\in X$ write
$u'=\pi_Y(u)$ and $v'=\pi_Y(v)$. Suppose $h$ is the
hyperplane separating $u$ and $v$.
\begin{enumerate}
\item If $h\cap Y=\emptyset$, then $u'=v'$.
\item If $h\cap Y\neq\emptyset$, then $u'$ and $v'$ are
  adjacent vertices in $Y$. Moreover, the hyperplane
  separating $u'$ and $v'$ is exactly $h$.
\end{enumerate}
\end{lem}

\begin{proof}
  Suppose without loss of generality that
  $d(v,Y^{(0)})\le d(u,Y^{(0)})$. Let $\omega_v$ and
  $\omega_u$ be the combinatorial geodesics which realize
  the $\ell^1$-distance from $v$ to $Y^{(0)}$ and
  $u$ to $Y^{(0)}$ respectively. 

  Suppose first that $h\cap Y=\emptyset$. Then
  $h\cap \omega_v=\emptyset$, otherwise we would have
  $d(v,Y^{(0)})>d(u,Y^{(0)})$. Thus $h$ separates $u$ from
  $Y$. Moreover, each hyperplane dual to an edge in
  $\omega_v$ separates $u$ from $Y$. By Corollary \ref{c:2},
  we have $d(v,Y^{(0)})+1\le d(u,Y^{(0)})$. On the other
  hand, the concatenation of the edge $\overline{uv}$ with
  $\omega_v$ has length $\le d(v,Y^{(0)})+1$. Thus this
  concatenation realizes the $\ell^1$-distance from
  $u$ to $Y^{(0)}$. It follows that $u'=v'$.

  Now suppose $h\cap Y\neq\emptyset$. First, by Lemma
  \ref{l:45} we get
  $\omega_v\cap h=\omega_u\cap h=\emptyset$ because
  otherwise $h$ would be dual to some edge in $\omega_v$ or
  $\omega_u$ and thus separate $u$ or $v$ from $Y$ and hence
  be disjoint from $Y$. Let $\omega$ be
  a geodesic joining $v'$ and $u'$. Note that $\omega$ is
  contained in $Y$. The
  path obtained by concatenating $\omega_v$, $\omega$ and  $\omega_u$ must intersect $h$ because $v$
  and $u$ lie on different sides of $h$. Thus $h$ must
  intersect $\omega$ and thus
  separate $v'$ and $u'$. To see that $v'$ and $u'$ are
  adjacent, it is enough to show that $h$ is the only
  hyperplane separating $u'$ and $v'$. Note, however, that
  if $h'$ is a hyperplane separating $u'$ from $v'$, then
  $h'$ must intersect the path obtained by contatenating
  $\omega_v$, the edge from $v$ to $u$ and $\omega_u$. By
  Lemma \ref{l:45} we get
  $h'\cap\omega_v=h'\cap\omega_u=\emptyset$ as above. Thus,
  $h'$ intersects the edge from $u$ to $v$ and hence $h'=h$.

\end{proof}

The above lemma implies that we can naturally extend the nearest
point projection map $\pi_Y:X^{(0)}\to Y^{(0)}$ to
$\pi_Y:X^{(1)}\to Y^{(1)}$.  The next result follows from
from Lemma \ref{l:5}.

\begin{cor}
\label{c:1}
Let $Y\subseteq X$ be a convex subcomplex. Let
$\omega\subseteq X$ be a combinatorial geodesic. Then
$\pi_{Y}(\omega)$ is also a combinatorial geodesic.
\end{cor}

Note that it is possible that $\pi_Y(\omega)$ is a single point.

\section{The geodesics lemma}
Throughout this section, $X$ will be a uniformly locally
finite Gromov-hyperbolic CAT$(0)$ cube complex. Let
$\partial X$ be the boundary of $X$.

\begin{defn}
Let $x\in X$ be a vertex and let $\eta\in \partial X$.
  Define the interval:
  \[
  \Geo{x}{\eta} = \{y\in X^{(0)}:\text{$y$ lies on a
    combinatorial geodesic from $x$ to $\eta$}\}
    \]
\end{defn}

Recall that if $X$ is $\delta$-hyperbolic, then for any $x\in X$ and $\eta\in \partial X$,
any two combinatorial geodesic ray $\omega_1$ and $\omega_2$
from $x$ to $\eta$ satisfy 
$d(\omega_1(t),\omega_2(t))\leq2\delta$ for each $t\geq 0$ and
$d_{H}(\omega_1,\omega_2)\leq2\delta$. Here
$d_{H}(\omega_1,\omega_2)$ denotes the Hausdorff distance
between $\omega_1$ and $\omega_2$.


Now we will prove Lemma \ref{modulo.finite}. Note that it suffices to prove the case when $x$ and $y$ are
adjacent. Thus in the rest of this section, we will assume $x$ and
$y$ are two adjacent vertices in $X$.

\begin{lem}
\label{l:6}
Let $h$ be the hyperplane separating $x$ and $y$ and let
$\overline{y\eta}$ be a combinatorial geodesic ray from $y$
to $\eta$.
\begin{enumerate}
\item If $\overline{y\eta}$ never crosses $h$, then each
  vertex of $\overline{y\eta}$ in contained in $\Geo{x}{\eta}$.
\item If $\overline{y\eta}$ crosses $h$, let
  $z\in \overline{y\eta}$ be the first vertex after
  $\overline{y\eta}$ crosses $h$ and let
  $\overline{z\eta}\subseteq \overline{y\eta}$ be the ray
  after $z$. Pick a combinatorial geodesic segment
  $\overline{xz}$. Then $\overline{xz}$ and
  $\overline{z\eta}$ fit together to form a combinatorial
  geodesic ray. In particular, each vertex of
  $\overline{z\eta}$ is contained in $\Geo{x}{\eta}$.
\end{enumerate}
\end{lem}

\begin{proof}
  To see (1), let $\overline{xy}$ be the edge joining $x$
  and $y$. Then $h$ is the hyperplane dual to
  $\overline{xy}$. Since $\overline{y\eta}$ never crosses $h$,
  then each hyperplane which is dual to some edge of
  $\overline{y\eta}$ is different from $h$. Thus the
  concatenation of $\overline{xy}$ and $\overline{y\eta}$ is
  a combinatorial geodesic ray because all hyperplanes dual
  to its edges are distinct \cite[Lemma
  13.1]{MR2377497}. Thus each vertex of $\overline{y\eta}$
  in contained in $\Geo{x}{\eta}$.

  Now we prove (2). Since $\overline{y\eta}$ is a
  combinatorial geodesic ray, it follows \cite[Lemma
  13.1]{MR2377497} that $\overline{z\eta}$ does not cross
  $h$. Suppose the concatenation of $\overline{xz}$ and
  $\overline{z\eta}$ is not a combinatorial geodesic
  ray. Since $\overline{xz}$ and $\overline{z\eta}$ are
  already geodesic, the only possibility is that there exist
  edges $e_1\subseteq\overline{xz}$ and
  $e_2\subseteq\overline{z\eta}$ such that they are dual to
  the same hyperplane $h'$, again by \cite[Lemma
  13.1]{MR2377497}. Let $u_i$ and $v_i$ be endpoints of
  $e_i$ indicated in the picture below.
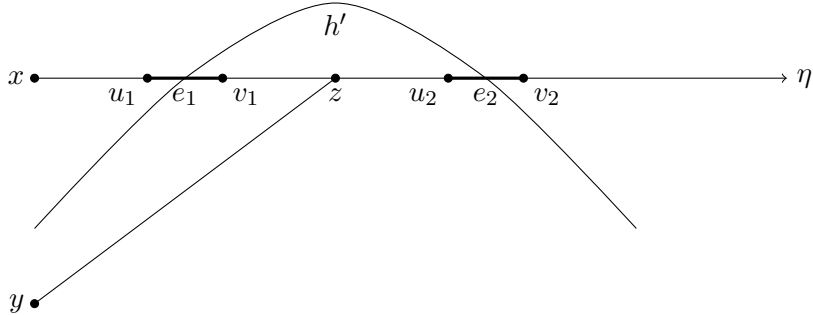
\[
  \begin{tikzpicture}
    \coordinate (x) at (0,0);
    \coordinate (h1) at (0,-2);
    \coordinate (y) at (0,-3);
    \coordinate (u1) at (1.5,0);
    \coordinate (e1) at (2,0);
    \coordinate (v1) at (2.5,0);
    \coordinate (z) at (4,0);
    \coordinate (h) at (4,1);
    \coordinate (u2) at (5.5,0);
    \coordinate (e2) at (6,0);
    \coordinate (v2) at (6.5,0);
    \coordinate (h2) at (8,-2);
    \coordinate (eta) at (10,0);
    \filldraw (x) circle (1.5pt) node[left] {$x$};
    \filldraw (y) circle (1.5pt) node[left] {$y$};
    \filldraw (u1) circle (1.5pt) node[below left] {$u_1$};
    \node[below] at (e1) {$e_1$};
    \filldraw (v1) circle (1.5pt) node[below right] {$v_1$};
    \filldraw (z) circle (1.5pt) node[below] {$z$};
    \node[below] at (h) {$h'$};
    \filldraw (u2) circle (1.5pt) node[below left] {$u_2$};
    \node[below] at (e2) {$e_2$};
    \filldraw (v2) circle (1.5pt) node[below right] {$v_2$};
    \node[right] at (eta) {$\eta$};
    \draw [->] (x) -- (eta);
    \draw (y) -- (z);
    \draw [very thick] (u1) -- (v1);
    \draw [very thick] (u2) -- (v2);
    \draw plot [smooth] coordinates {(h1) (e1) (h) (e2) (h2)};
  \end{tikzpicture}
  \]
  Since $\overline{xz}$ is a combinatorial geodesic, it
  crosses $h'$ only once (\cite[Lemma
  13.1]{MR2377497}). Thus the segments $\overline{xu_1}$ and
  $\overline{v_1z}$ stay in different sides of $h'$. In
  particular, $x$ and $z$ are in different sides of
  $h'$. Since $\overline{y\eta}$ is a combinatorial geodesic
  ray, it crosses $h'$ only once, thus the segment
  $\overline{yz}\cup\overline{zu_2}$ is in one side of
  $h'$. In particular, $y$ and $z$ are in the same side of
  $h'$. Thus we deduce that $x$ and $y$ are separated by
  $h'$. Since $x$ and $y$ are adjacent, there is only one
  hyperplane separating them, thus $h'=h$. This is a
  contraction since $\overline{z\eta}$ does not cross $h$.
\end{proof}

\begin{proof}[Proof of Lemma \ref{modulo.finite}]
  We assume $x$ and $y$ are adjacent. Let $h$ be the
  hyperplane separating them. We argue by contradiction and suppose
  there exists a sequence $\{z_i:i\in\mathbb{N}\}$ in
  $\Geo{y}{\eta}\setminus \Geo{x}{\eta}$ with $z_i\neq z_j$ for
  $i\neq j$. Since $X$ is uniformly locally finite, we can
  assume $d(z_i,y)\to\infty$ as $i\to \infty$. Let
  $\omega_i$ be a combinatorial geodesic segment from $y$
  to $\eta$ such that $z_i\in\omega_i$. By Lemma \ref{l:6},
  each $\omega_i$ crosses $h$. Let
  $\overline{yv_i}\subseteq\omega_i$ be the segment before
  $\omega_i$ crosses $h$, and let
  $\overline{u_i\eta}\subseteq\omega_i$ be the segment after
  $\omega_i$ crosses $h$ (see the picture below). It follows
  from Lemma \ref{l:6} (2) that
  $z_i\subseteq\overline{yv_i}$. In particular,
  $d(v_i,y)\to\infty$ as $i\to \infty$.
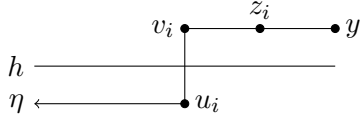
\[
  \begin{tikzpicture}
    \coordinate (eta) at (0,-0.5);
    \coordinate (h1) at (0,0);
    \coordinate (ui) at (2,-0.5);
    \coordinate (vi) at (2,0.5);
    \coordinate (zi) at (3,0.5);
    \coordinate (y) at (4,0.5);
    \coordinate (h2) at (4,0);
    \node[left] at (eta) {$\eta$};
    \node[left] at (h1) {$h$};
    \filldraw (ui) circle (1.5pt) node[right] {$u_i$};
    \filldraw (vi) circle (1.5pt) node[left] {$v_i$};
    \filldraw (zi) circle (1.5pt) node[above] {$z_i$};
    \filldraw (y) circle (1.5pt) node[right] {$y$};
    \draw (h1) -- (h2);
    \draw [->] (y) -- (vi) -- (ui) -- (eta);
  \end{tikzpicture}
  \]

  Recall that $h$ gives rise to two combinatorial
  hyperplanes, one containing $x$, which we denote by $h_x$,
  and one containing $y$, which we denote by $h_y$. Note
  that $v_i\in h_y$ by construction. Since $h_y$ is a
  convex subcomplex, it follows \cite[Proposition
  13.7]{MR2377497} that $\overline{yv_i}\subseteq h_y$. Since
  $X$ is uniformly locally finite (hence locally compact)
  and $d(v_i,y)\to\infty$, up to passing to a subsequence,
  we can assume the sequence of segments
  $\{\overline{yv_{i}}\}_{i=1}^{\infty}$ converges to a
  combinatorial geodesic ray $\omega$. Since
  $\overline{yv_i}\subseteq h_y$ for each $i$,
  $\omega\subseteq h_y$. Moreover, by $\delta$-hyperbolicity, the Hausdorff distance between $\omega$ and
  any of $\omega_i$ is less than $2\delta$. Thus $\omega$ is a
  combinatorial geodesic ray joining $y$ and $\eta$. Since
  $\omega$ is contained in $h_y$ we get that for every
  $i$ and every vertex $w\in\omega_i$ we have
  $d(w,h_y)\leq2\delta$.

  Let $\pi: X^{(1)}\to h^{(1)}_y$ be the nearest point
  projection from $X^{(1)}$ to the 1-skeleton of convex
  subcomplex $h_y$. Then $\pi(\omega_i)$ is a combinatorial
  geodesic by Corollary \ref{c:1}. It follows from the above
  remarks and the
  definition of $\pi$ that $d_H(\omega_i,\pi(\omega_i))\leq2\delta$

  Thus $\pi(\omega_i)$ is a combinatorial geodesic ray
  joining $y$ and $\eta$. Since
  $\overline{yv_i}\subseteq h_y$,
  $\pi(\overline{yv_i})=\overline{yv_i}$. Thus
  $\overline{yv_i}$ is contained in $\pi(\omega_i)$. In
  particular, $z_i\in \pi(\omega_i)$. Since
  $\pi(\omega_i)\subseteq h_y$, it never crosses $h$, thus Lemma
  \ref{l:6} (1) implies $z_i\in \Geo{x}{\eta}$, which is a
  contradiction.
\end{proof}

\section{Finite Borel equivalence relations}

We will use the following standard application of the second
reflection theorem \cite[Theorem 35.16]{kechris}. Below, if
$E$ is an equivalence relation of $Z$ and $A\subseteq Z$, then
$E|A$ denotes $E\cap A\times A$.

\begin{lem}\label{reflection}
  Let $Z$ be a Polish space, $A\subseteq Z$ be analytic and let $E$ be an analytic equivalence relation on $Z$
  such that there is some $n >1$ such that every
  $E|A$-class has size less than $n$.
  Then there is a Borel equivalence relation $F$ on $Z$
  with $E|A\subseteq F$ such that every $F$-class has size less than $n$.
\end{lem}
\begin{proof}
  Note that $G = E|A\cup\{(z,z):z\in Z\}$ is an analytic equivalence relation on $Z$
  whose classes have size less than $n$.
  Now consider $\Phi\subseteq\Pow(Z^2)\times\Pow(Z^2)$ defined as follows:
  \begin{align*}
    (B,C)\in\Phi \iff
    & \forall x\,\,\lnot x\mathrel{C}x \\
    & \land\forall(x,y)\,\lnot x\mathrel{B}y\lor\lnot y\mathrel{B}x \\
    & \land\forall(x,y,z)\,\lnot x\mathrel{B}y\lor\lnot y\mathrel{B}z\lor\lnot x\mathrel{C}z \\
    & \land\forall_{i=1}^n x_i\,(\bigvee_{i\neq j} x_i = x_j)
    \lor(\bigvee_{i\neq j}\lnot x_i \mathrel{B} x_j)
  \end{align*}
  Note that $\Phi(B,B^c)$ holds iff $B$ is an equivalence relation on $Z$
  whose classes have size less than $n$,
  so in particular we have $\Phi(G,G^c)$.
  Now $\Phi$ is $\mathbf{\Pi}^1_1$ on $\mathbf{\Pi}^1_1$,
  hereditary and continuous upward in the second variable,
  so by the second reflection theorem \cite[Theorem 35.16]{kechris},
  there is a Borel set $F\supset G$ such that $\Phi(F,F^c)$ holds,
  and we are done.
\end{proof}

\section{Proof of main theorem}

The following fact lets us reduce our problem to the case of free actions.
\begin{lem}\label{finindex}
  Every cubulated hyperbolic group has a finite index subgroup acting freely and cocompactly
  on a $\CAT(0)$-cube complex.
\end{lem}
\begin{proof}
  If $G$ is a hyperbolic group acting properly and cocompactly on a $\CAT(0)$ cube complex $X$,
  then by Agol's theorem \cite[Theorem 1.1]{agol} (see also  Wise \cite{wise})
  there is a finite index subgroup $F$ acting faithfully and
  specially on $X$ (see Haglund and Wise \cite[Definition 3.4]{MR2377497} for the
  definition of special action).
  Now $F$ embeds into a right-angled Artin group which is torsion-free,
  so $F$ is torsion-free.
  Since every stabilizer is finite by properness of the action,
  it must be trivial since $F$ is torsion-free,
  and thus $F$ acts freely on $X$.
\end{proof}
 

\begin{proof}[Proof of Theorem \ref{abstractthm}]
  
Let $V$ be the set of vertices of the graph. Note that $V$
as a metric space is hyperbolic since the action of $G$ is
geometric. Below, by $\partial V$ we denote the Gromov
boundary of $V$.
Fix $v_0\in V$
and fix a total order on $V$ such that $d(v_0,v)\le
d(v_0,w)\implies v\le w$, where $d$ denotes the graph
distance on $V$.
Fix a transversal $\tilde{V}$ of the action of $G$ on $V$
(the transversal is finite since the action is cocompact).
For $v\in V$,
we denote by $\tilde{v}$ the unique element of $\tilde{V}$ in the orbit of $v$.
By a directed edge of $V$ we mean a pair $(v,v')\in V^2$ such that there is an edge from $v$ to $v'$.
We colour the directed edges of $V$ as follows.
We assign a distinct colour to every directed edge $(v,v')$ with $v\in\tilde{V}$,
and this extends uniquely (by freeness) to a $G$-invariant colouring on all directed edges.
Let $C$ be the set of colours (which is finite since $V$ is locally finite),
and let $c(v,v')$ be the colour of $(v,v')$.
Fix any total order on $C$.
This induces a lexicographical order on $C^{<\N}$ (the set
of all finite sequences of elements of $C$).

For any combinatorial geodesic $\eta\in V^{<\N}$ and $m,n\in\N$,
define:
\[
  c(\eta,m,n)
  = (c(\eta_m,\eta_{m+1}),c(\eta_{m+1},\eta_{m+2}),\ldots,c(\eta_{m+n-1},\eta_{m+n}))
  \in C^{<\N}
  \]
For every $a\in\partial V$,
define $S^a\subseteq V\times C^{<\N}$ as follows:
\begin{align*}
 S^a = \{(\eta_m,c(\eta,m,n))\in V\times C^{<\N}:
  \eta \mbox{ is a combinatorial geodesic}\\
  \mbox{ from } v_0 \mbox{ to }a \mbox{ and } m,n\in\N\}
 \end{align*}
Let $s^a_n\in C^{<\N}$ be the least string of length $n$
which appears infinitely often in $S^a$,
ie. such that there are infinitely many $v\in V$ for which $(v,s^a_n)\in S^a$.
Note that each $s^a_n$ is an initial segment of $s^a_{n+1}$.
Let $$T^a_n = \{v\in V:(v,s^a_n)\in S^a\}$$
and let $v^a_n = \min T^a_n$ (with respect to the ordering on $V$).
Note that every vertex in $T^a_n$ has an edge coloured by $s^a_1$ leaving it,
so every vertex of $T^a_n$ is in the same orbit.
Let $$k^a_n = d(v_0,v^a_n)$$ and note that $k^a_n$ is nondecreasing in $n$.

Now let $Z = \{a\in\partial V:k^a_n\not\to\infty\}$.
Then for each $a\in\partial V$,
since $k^a_n\not\to\infty$ and $V$ is discrete,
there is a finite set containing all $v^a_n$,
so there is some $v\in V$ which is in $T^a_n$ for infinitely many $n$.
Thus the geodesic class determined by the combinatorial geodesic
starting at $\tilde{v}$ (which is determined by $k^a_1$)
and following the colours of $\lim_n s^a_n\in C^\N$ is a Borel selector.
Thus $E$ is smooth on the saturation $[Z]_E$.

Now let $Y = partial V\setminus [Z]_E = \{a\in\partial X:\forall
bEa\ k^b_n\to\infty\}$.
We will show that $E$ is hyperfinite on $Y$.
For each $n\in\N$,
and define $H_n:\partial V\to 2^V$
by $$H_n(a) = g^a_n T^a_n,$$
where $g^a_n\in G$ is the unique element with $g^a_n v^a_n\in\tilde{V}$.
Let $E_n$ be the equivalence relation on $\im H_n$
which is the restriction of the shift action of $G$ on $2^V$.
We have the following lemma:
\begin{lem}
  There exists $K\in\N$ such that on $\im H_n$ the
  relation $E_n$ has equivalence classes of size at most
  $K$.
\end{lem}
\begin{proof}
  Let $a,b\in\partial V$ and suppose $g\in G$ is such that $gH_n(a) = H_n(b)$,
  i.e. $gg^a_n T^a_n = g^b_n T^b_n$.
  Since the vertices in both sets are in the same orbit,
  $g^a_n v^a_n$ and $g^b_n v^b_n$ are elements of $\tilde{V}$
  which are in the same orbit,
  so they are equal,
  say to some $v\in\tilde{V}$.
  It suffices to show that $d(v,gv)\le 6\delta$,
  since then we can take choose any $K\in\omega$
  larger than $\max_{v\in\tilde{V}}|\{g:d(v,gv)\le 6\delta\}|$.
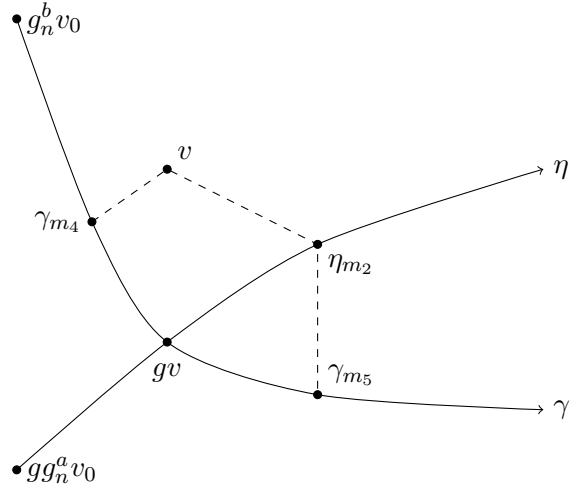
\[
  \begin{tikzpicture}
    \coordinate (a) at (0,-3);
    \coordinate (b) at (0,3);
    \coordinate (m4) at (1,0.3);
    \coordinate (v) at (2,1);
    \coordinate (gv) at (2,-1.3);
    \coordinate (m2) at (4,0);
    \coordinate (m5) at (4,-2);
    \coordinate (eta) at (7,1);
    \coordinate (gamma) at (7,-2.2);
    \coordinate (c) at (8,0);
    \filldraw (a) circle (1.5pt) node[right] {$gg^a_n v_0$};
    \filldraw (b) circle (1.5pt) node[right] {$g^b_n v_0$};
    \filldraw (m4) circle (1.5pt) node[left] {$\gamma_{m_4}$};
    \filldraw (v) circle (1.5pt) node[above right] {$v$};
    \filldraw (gv) circle (1.5pt) node[yshift=-0.4cm] {$gv$};
    \filldraw (m2) circle (1.5pt) node[below right] {$\eta_{m_2}$};
    \filldraw (m5) circle (1.5pt) node[above right] {$\gamma_{m_5}$};
    \draw (eta) node[right] {$\eta$};
    \draw (gamma) node[right] {$\gamma$};
    \draw [->] plot [smooth] coordinates {(a) (gv) (m2) (eta)};
    \draw [->] plot [smooth] coordinates {(b) (m4) (gv) (m5) (gamma)};
    \draw [dashed] (m4) -- (v) -- (m2) -- (m5);
  \end{tikzpicture}
\]

  Note that since $T^a_n$ and $T^b_n$ are infinite,
  we have that $gg^a_n a = g^b_n b$,
  which we will call $c\in\partial X$.
  Let $\eta$ be a geodesic from $gg^a_nv_0$ to $c$ with $\eta_{m_1} = gv$.
  Now $v\in gg^a_n T^a_n$,
  so there is some $m_2$ with $d(v,\eta_{m_2})\le 2\delta$.
  Note that by choice of $v^a_n$,
  we have $m_2\ge m_1$.
  Now let $\gamma$ be a geodesic from $g^b_n\hat{x}$ to $c$ with $\gamma_{m_3} = gv$.
  By the choice of $v^b_n$,
  there is some $m_4\le m_3$ such that $d(v,\gamma_{m_4})\le 2\delta$.
  Also $\eta$ and $\gamma$ are $2\delta$-close after they go through $gv$,
  so since $m_2\ge m_1$,
  there is some $m_5\ge m_3$ such that $d(\eta_{m_2},\gamma_{m_5})\le 2\delta$.
  Thus
  \begin{align*}
    2d(v,gv)
    & \le d(v,\gamma_{m_4}) + d(\gamma_{m_4},gv) +
      d(v,\eta_{m_2}) + d(\eta_{m_2},\gamma_{m_5})+ d(\gamma_{m_5},gv) \\
    & = d(\gamma_{m_4},\gamma_{m_5}) + d(v,\gamma_{m_4}) +
      d(v,\eta_{m_2}) + d(\eta_{m_2},\gamma_{m_5}) \\
    & \le 2(d(v,\gamma_{m_4}) + d(v,\eta_{m_2}) +d(\eta_{m_2},\gamma_{m_5})) \\
    & \le 2(6\delta),
  \end{align*}
where the first equality follows from the fact that $\gamma$
is a geodesic.
\end{proof}
Now $\im H_n$ is analytic,
so by Lemma \ref{reflection},
there is a Borel equivalence relation $E'_n$ on $2^V$ containing $E_n$
whose classes are of size at most $K$.
Let $f_n:2^V\to 2^\N$ be a reduction for $E'_n\le_B\id_{2^\N}$,
and define $f:\partial V\to (2^\N)^\N$ by $f(a) =
(f_n(H_n(a)):n\in\N)$. Write $E'$ for the pullback of $E_1$ via
$f$. Note that since each $E_n'$ is finite, the relation
$E'$ is countable. As $E'$ is clearly hypersmooth, we get
that $E'$ is hyperfinite by \cite[Theorem 8.1.5]{gao}. Now, $f$ is a homomorphism
from $E$ to $E_1$. Indeed, if $a,b\in\partial V$ with $aEb$,
then by Lemma \ref{modulo.finite},
there is $N\in\N$ such that $H_n(a) E_n H_n(b)$ for $n\ge N$,
and thus $f(x)E_1 f(y)$.
Thus, $E\subseteq E'$ is a subrelation of a hyperfinite one,
and hence it hyperfinite as well.
\end{proof}

\begin{proof}[Proof of Theorem \ref{main}]
 
Let $G$ be a cubulated $\delta$-hyperbolic group.
Since hyperfiniteness passes to finite-index extensions
\cite[Proposition 1.3]{jkl},
by Lemma \ref{finindex}, we can assume that $G$ acts freely
and cocompactly on a $\CAT(0)$ cube complex $X$. 
Let $V = X^{(0)}$ be the set of vertices of $X$.  Now the
statement follows from Theorem \ref{abstractthm}.

\end{proof}

\bibliographystyle{alpha}
\bibliography{1}

\end{document}